\documentclass[11pt]{article} 
\usepackage{graphicx,amsthm,amsmath,amscd,amssymb,hyperref}
\usepackage[width=13cm,height=20cm]{geometry}
\usepackage{ucs}
\usepackage[utf8x]{inputenc}
\usepackage{tikz}
\usepackage{color}
\definecolor{tocolor}{rgb}{.1,.1,.1}
\definecolor{urlcolor}{rgb}{.2,.2,.6}
\definecolor{linkcolor}{rgb}{.1,.1,.5}
\definecolor{citecolor}{rgb}{.4,.2,.1}
\hypersetup{backref=true, colorlinks=true, urlcolor=urlcolor, linkcolor=linkcolor, citecolor=citecolor}
\newtheorem{thm}{Theorem}[section]
\newtheorem{cor}[thm]{Corollary}
\newtheorem{lem}[thm]{Lemma}

\theoremstyle{definition}
\newtheorem{defn}[thm]{Definition}

\theoremstyle{remark}
\newtheorem{rem}[thm]{Remark}
\numberwithin{equation}{section}

\newcommand{\blank}[1]{}

\newcommand{\JJ}{\mathbb J}
\newcommand{\R}{\mathbb R}
\newcommand{\C}{\mathbb C} 

\newcommand{\N}{\mathbb N}
 
\newcommand{\Lie}{\mathcal{L}}

\newcommand{\I}{\mathcal{I}}

\renewcommand{\phi}{\varphi} 
\renewcommand{\to}{\longrightarrow}

\newcommand{\oto}[1]{\overset{#1}\to}

\renewcommand{\mapsto}{\longmapsto}

\newcommand{\comp}{\circ}
\newcommand{\sr}{\mathcal} 
 
\renewcommand{\O}{\mathcal{O}}

\renewcommand{\Re}{\textrm{Re}} 
\newcommand{\del}{\partial}

\renewcommand{\^}{\wedge}

\renewcommand{\epsilon}{\varepsilon}
\newcommand{\hide}[1]{}

\newcommand{\Id}{\textrm{Id}}
\newcommand{\LP}{\mathfrak{L}}

\renewcommand{\Im}{\textrm{Im}}

\usepackage{titlesec}

\titleformat*{\section}{\large\bfseries}
\titleformat*{\subsection}{\large\bfseries}
\titleformat*{\subsubsection}{\large\bfseries}

\begin{document}

\title{\Large\sffamily Local analytic geometry of generalized complex structures}
\author{Michael Bailey\footnote{\href{mailto:M.A.Bailey@uu.nl}{\texttt{M.A.Bailey@uu.nl}}} \and Marco Gualtieri\footnote{\href{mailto:mgualt@math.toronto.edu}{\texttt{mgualt@math.toronto.edu}}}}
\date{}
\maketitle
\renewcommand{\abstractname}{\vspace{-7ex}}

\begin{abstract}
  A generalized complex manifold is locally gauge-equivalent to the product of a holomorphic Poisson manifold with a real symplectic manifold, but in possibly many different ways.  In this paper we show that the isomorphism class of the holomorphic Poisson structure occurring in this local model is independent of the choice of gauge equivalence, and is hence the unique local invariant of generalized complex manifolds.  This completes the local classification of generalized complex structures.
  We use this result to prove that the complex locus of a generalized complex manifold naturally inherits the structure of a complex analytic space. 
\end{abstract}

\section{Introduction}\label{gc definitions}

A generalized complex structure $\JJ$ on a smooth manifold $M$ is a complex structure on the bundle $TM\oplus T^*M$ that is involutive for the Courant bracket~\cite{MR2013140,MR2811595}.  
% Usually one includes an additional ``twisting'' by a closed 3-form, but we may ignore this issue here since we are only concerned with the local classification of this type of geometric structure.
Two such structures are isomorphic when they are related by a Courant automorphism, which is the composition of a diffeomorphism of $M$ with a bundle automorphism of $TM\oplus T^*M$ induced by a closed 2-form known as a B-field gauge transformation.  The gauge transformation induced by the closed 2-form $B$ is given by
\[
e^B\cdot(X+\xi) = X +  \xi+ i_X B,\qquad X+\xi\in TM\oplus T^*M.
\]
The simplest examples of generalized complex structures are those induced by a usual complex structure $I$ or a symplectic structure $\omega$, and have the form
\begin{equation}\label{usuform}
\JJ_I = \begin{pmatrix}-I & 0\\ 0& I^*\end{pmatrix},\qquad \JJ_\omega = \begin{pmatrix}0 & -\omega^{-1}\\\omega  & 0\end{pmatrix}.
\end{equation}
The example we shall focus on in this paper is induced by a holomorphic Poisson structure $(I,\sigma)$, consisting of a complex structure $I$ and a holomorphic bivector field $\sigma$ satisfying the Poisson condition $[\sigma,\sigma]=0$. Decomposing into real and imaginary parts, we have $\sigma = -\tfrac{1}{4}(IQ + i Q)$, for $Q = 4I\Re(\sigma)$ a real Poisson structure; the induced generalized complex structure is then
\begin{equation}\label{holpo}
\JJ_{\sigma}=  \begin{pmatrix}
   -I & Q\\
    0 & I^*
  \end{pmatrix}.
\end{equation}
The appearance of the real Poisson structure $Q$ is a general phenomenon: for any generalized complex structure $\JJ$, the bundle map $\pi_{TM}\circ \JJ|_{T^*M} : T^*M\to TM$ defines a Poisson structure $Q_\JJ$ whose rank partially controls the local geometry.  

It was shown recently that generalized complex manifolds are locally isomorphic to a product of the examples above.  Building on work in~\cite{MR2811595} and~\cite{MR2240211} on the local classification problem, Bailey obtained the following result. 
\begin{thm}[\cite{Bailey}]\label{bail}
  Let $(M,\JJ)$ be a generalized complex $2n$-manifold, and let $p$ be a point where $Q_\JJ$ has rank $2n-2k$. Then there exists a holomorphic Poisson structure $\sigma$ defined on a neighbourhood $U$ of the origin in $\C^k$ and vanishing at zero, such that at $p$, $(M,\JJ)$ is locally isomorphic to $(U,\JJ_\sigma)\times (\R^{2n-2k},\JJ_\omega)$ at $(0,0)$, where $\omega$ is the Darboux symplectic form.
\end{thm}
It does not follow from the above theorem that the holomorphic isomorphism class of the germ of the Poisson structure $\sigma$ at zero is uniquely determined by $\JJ$. This is because, as explained in~\cite{Gualtieri:2007fe}, one may generally find gauge transformations relating two holomorphic Poisson structures $(I,\sigma)$, $(J,\tau)$ which are not holomorphically equivalent.
That is, one may find a real closed 2-form $B$ such that 
\begin{equation}\label{gageq}
e^B
  \begin{pmatrix}
    -I & Q\\
    0 & I^*
  \end{pmatrix}
e^{-B}
=
  \begin{pmatrix}
    -J & Q\\
    0 & J^*
  \end{pmatrix},
\end{equation}
where the imaginary parts of $\sigma$ and $\tau$, forced to coincide by the above condition, are denoted by $Q$.  In general, Equation~\ref{gageq}, while it does express isomorphism as generalized complex structures, does not imply the existence of a biholomorphic map taking $(I,\sigma)$ to $(J,\tau)$.

Despite this concern, we shall prove in Corollary~\ref{mainres} that two germs of holomorphic Poisson structures near a point $p$ which vanish at $p$ are holomorphically equivalent if and only if their induced generalized complex structures are isomorphic.  An immediate corollary is that the local structure of a generalized complex manifold is completely characterized by the holomorphic equivalence class of such a germ.  

We say that points where the real Poisson structure $Q_\JJ$ vanishes are of \emph{complex type}, since, at these points, up to $B$-transform, $\JJ$ takes the form $\JJ_I$ in \eqref{usuform}.  By  Theorem~\ref{bail}, the locus of points of complex type can be described locally as the zero set of a holomorphic Poisson structure.  As an application of our results, we prove in Theorem~\ref{schem} that the analytic structure inherited by the complex locus by this local description is globally well-defined, rendering it into a complex analytic space. 

\vspace{1ex}
\noindent\emph{Acknowledgements:} We thank Gil Cavalcanti and Pierre Milman for helpful discussions.  This research was supported by an NSERC Discovery grant.

\section{Interpolation of holomorphic Poisson structures}\label{families of holomorphic poisson}

Our main technical tool, proven in Appendix \ref{SCI section}, is the
following parametrized version of Theorem \ref{bail}:

\begin{thm}\label{families prop}
  Let $\JJ_t$, $t\in [0,1]$ be a smooth family of generalized complex
  structures on a manifold $M$ that are all of complex type at the
  point $p\in M$ (i.e., $Q_{\JJ_t}$ vanishes at $p$).  Then, in a neighbourhood of $p$, there is a smooth
  family of gauge transformations by closed 2-forms $B_t$ which
  renders $\JJ_t$ isomorphic to a smooth family of holomorphic Poisson
  structures $(I_t,\sigma_t)$ with $\sigma_t(p)=0$ for all $t$.

  Furthermore, if $\JJ_0$ and $\JJ_1$ are already of holomorphic
  Poisson type~\eqref{holpo}, then the family $B_t$
  may be chosen such that $B_0=B_1=0$.
\end{thm}

We use this result as follows.  Suppose that $\JJ$ is a generalized complex structure on $M$ and that $p\in M$ has complex type.  Suppose that we have two different gauge transformations by 2-forms $B_0$ and $B_1$ rendering $\JJ$ isomorphic to $\JJ_{\sigma_0}$ and $\JJ_{\sigma_1}$ respectively, where $(I_0,\sigma_0)$ and $(I_1,\sigma_1)$ are holomorphic Poisson structures with $\sigma_0(p)=\sigma_1(p)=0$ (where the existence of such gauge transforms are guaranteed by Theorem~\ref{bail}).  It follows immediately that the gauge transformation by the two-form $B = B_1-B_0$ satisfies
\[
e^B \JJ_{\sigma_0}e^{-B} = \JJ_{\sigma_1},
\]
and more importantly, we may scale the gauge action to obtain a smooth family of generalized complex structures
\[
\JJ_t = e^{tB}\JJ_{\sigma_0}e^{-tB},
\]
which interpolate between $\JJ_{\sigma_0}$ and $\JJ_{\sigma_1}$.  

We may now apply Theorem~\ref{families prop} to obtain a family of 2-forms $\tilde B_t$ that transforms $\JJ_t$ into an interpolating family of holomorphic Poisson structures, i.e. 
\[
e^{\tilde B_t}\JJ_t e^{-\tilde B_t} = \JJ_{\sigma_t},
\]
where $(I_t,\sigma_t)$ is a family of holomorphic Poisson structures with $\sigma_t(p)=0$ for all $t$.
\begin{cor}\label{intgau}
Let $(I_0,\sigma_0)$ and $(I_1,\sigma_1)$ be holomorphic Poisson structures with $\sigma_0(p)=\sigma_1(p)=0$ and with gauge-equivalent associated generalized complex structures $\JJ_{\sigma_0}$, $\JJ_{\sigma_1}$.  Then there exists a family of gauge transformations $B_t, t\in [0,1]$, defined in a sufficiently small neighbourhood of $p$, such that 
\begin{equation}\label{intgaug}
e^{B_t}\JJ_{\sigma_0} e^{-B_t} = \JJ_{\sigma_t}
\end{equation}
is the generalized complex structure associated to a family $(I_t,\sigma_t), {t\in[0,1]}$, of holomorphic Poisson structures interpolating between the given pair.
\end{cor}

\section{Holomorphic equivalence from gauge equivalence}\label{holomorphic from gauge section}

Let $(I_t,\sigma_t)$, $t\in[0,1]$, be a smooth family of holomorphic Poisson structures, vanishing at point $p$, which, as in Corollary~\ref{intgau}, are all gauge-equivalent, in the sense that we have a smooth family of real closed 2-forms $B_t$ such that Equation~\ref{intgaug} holds.
\begin{lem}\label{gauge to hol}
In a sufficiently small neighbourhood of $p$, the family $(I_t,\sigma_t)$ is generated by the flow of a vector field, that is, there is a real vector field $X_t$ near $p$ such that $\dot I_t = \Lie_{X_t} I_t$ and $\dot \sigma_t = \Lie_{X_t} \sigma_t$.  $X_t$ is Hamiltonian for the real Poisson structure $Q = \Im(\sigma_t)$.
\end{lem}

\begin{proof}
Since $\sigma_t$ is determined by $Q$ and $I_t$, we need only prove the claim for $I_t$.  Explicitly, Equation~\eqref{intgaug} gives
\begin{equation*}
\begin{pmatrix}
  1 & 0\\ B_t & 1
\end{pmatrix}
  \begin{pmatrix}
    -I_0 & Q\\
    0 & I_0^*
  \end{pmatrix}
\begin{pmatrix}
  1 & 0\\ -B_t & 1
\end{pmatrix}
=
  \begin{pmatrix}
    -I_t & Q\\
    0 & I_t^*
  \end{pmatrix},
\end{equation*}
which is equivalent to the pair of equations, studied in~\cite{Gualtieri:2007fe}: 
  \begin{align}
     I_0 + QB_t &= I_t\label{one}\\
     I_0^* B_t + B_t I_t &=0.\label{two}
  \end{align}
Differentiating~\eqref{one}, we obtain the variation of the complex structure:
\begin{equation}\label{varcx}
\dot I_t =  Q \dot B_t.
\end{equation}
Differentiating~\eqref{two} and using~\eqref{one}, we obtain
\begin{align}\label{1,1 cond}
I_t^*\dot B_t + \dot B_t I_t = 0,
\end{align}
meaning that $\dot B_t$ is of type $(1,1)$ with respect to the complex structure $I_t$.  Since $B_t$, and hence $\dot B_t$, is also real and closed, we may find, in a sufficiently small neighbourhood of the point $p$, a smooth family of real-valued functions $f_t$ such that
\begin{align*}
\dot B_t &= i \bar\del_t \del_t f_t.
\end{align*}
In view of~\eqref{varcx}, this implies that 
\begin{align}
\dot I_t &= Q (i \bar\del_t \del_t f_t) \notag \\
&= 4I\Re(\sigma(i \bar\del_t \,d\, f_t)) \notag \\
&= 4I\Re(\bar\del_t \sigma(id f_t)) \qquad \textnormal{(since $\sigma$ is holomorphic)} \notag \\
&= 2I\Re(\bar\del_t (X_t)_{1,0}) \label{varcxf}
\end{align}
for $X_t = Q df_t$.  For a real vector field $X$, we have the fundamental formula
\begin{align}\label{ff}
\bar\del X_{1,0} = -\tfrac{1}{2} I (\Lie_X I).
\end{align}
From \eqref{varcxf} and \eqref{ff}, we have
\begin{align}
\dot I_t &= I \Re\left(I(\Lie_{X_t} I)\right) \notag \\
&= \Lie_{X_t} I.
\end{align}
\end{proof}

We now show that an analogue of Corollary~\ref{intgau} holds, where the the holomorphic Poisson structures are related, not by gauge equivalence, but by diffeomorphism.
\begin{thm}\label{hamthm}
  Let $(I_0,\sigma_0)$ and $(I_1,\sigma_1)$ be holomorphic Poisson
  structures in a neighbourhood of the point $p$, with
  $\sigma_0(p)=\sigma_1(p)=0$ and with gauge-equivalent associated
  generalized complex structures $\JJ_{\sigma_0}$, $\JJ_{\sigma_1}$.
  Then there exists a Hamiltonian flow $\varphi_t$, defined in a
  sufficiently small neighbourhood of $p$, such that
\begin{equation}
\varphi_t(I_0) = I_t\qquad\text{and}\qquad \varphi_t(\sigma_0)=\sigma_t,
\end{equation}
implying the holomorphic equivalence of $(I_t,\sigma_t)$ for all $t\in[0,1]$.
\end{thm}

\begin{proof}
Given the hypotheses, Corollary~\ref{intgau} puts us in the case of Lemma~\ref{gauge to hol}.  Since $Q$ vanishes at $p$, the flow
of $X_t$ is well-defined for all $t\in[0,1]$ in a sufficiently
small neighbourhood of $p$.  Therefore, the flow of the
time-dependent Hamiltonian vector field $X_t$ defines a family of
diffeomorphisms $\varphi_t$ taking $(I_0,\sigma_0)$ to
$(I_t,\sigma_t)$ for all $t$.
\end{proof}

Combining this result with Corollary~\ref{intgau}, we obtain our main result, which ensures that the holomorphic isomorphism class of the holomorphic Poisson structure $\sigma$ in Theorem~\ref{bail} is unique.
\begin{cor}\label{mainres}
Let $(M,\JJ)$ be a generalized complex manifold and let $\JJ$ be of complex type at $p\in M$. If the germ of $\JJ$ at $p$ is isomorphic to the generalized complex structure determined by each of two holomorphic Poisson germs $\sigma_0, \sigma_1$ at $p$, then $\sigma_0$ and $\sigma_1$ must be equivalent as holomorphic Poisson structures. 
\end{cor}

\section{Example}

Theorem~\ref{families prop} (and thus this whole paper), having at its heart a Nash-Moser type argument, does not give a reasonable construction.  However, we can see in a concrete case the phenomenon of gauge equivalence being realized as holomorphic equivalence.

Let $w,z$ be complex coordinates for $\C^2$, and let
$$\sigma = w \del_w \^ \del_z = \left(d\log w \^ dz \right)^{-1}$$
be a holomorphic Poisson bivector, and let $Q = 2i(\sigma - \bar\sigma) = -4\Im(\sigma)$ be the corresponding gauge-invariant real Poisson bivector.  The complex structure $I$ on $\C^2$ has canonical bundle generated by $dw\^ dz$.  We define a family of real, closed 2-forms, $B_t = i t dz \^ d\bar{z}$.

  Though it is not immediately obvious that $B_t$ will transform the generalized complex structure $\JJ_{I,\sigma}$ into a family, $\JJ_{I_t,\sigma_t}$, which is itself holomorphic Poisson, we can see this by observing how this example fits into the framework of Section \ref{holomorphic from gauge section}.
  
We specify a family, $w_t,z_t$, of holomorphic coordinates defining a family, $I_t$, of complex structures.  Let $z_t = z_0 = z$ be fixed, and let $w_t = w e^{it\bar{z}}$.  We observe that
\begin{align}\label{w dot}
\dot{w}_t = i\bar{z} w_t.
\end{align}
Since $z$ and $w_t$ should be holomorphic coordinates for $I_t$, we have that ${I^*_t dz = idz}$ and
\begin{align}\label{I on w}
I^*_t dw_t = i dw_t.
\end{align}
Differentiating \eqref{I on w}, and applying \eqref{w dot},
\begin{align}
\dot{I}^*_t dw_t + I^*_t d\dot{w}_t &= i d\dot{w}_t \notag \\
\dot{I}^*_t dw_t &= (i - I^*_t ) (i\bar{z}dw_t + iw_t d\bar{z}) \notag \\
 &= -2 w_t d\bar{z}. \label{dot I in example}
\end{align}
Along with $\dot{I}^*_t dz = 0$ and the reality condition on $\dot{I}^*_t$, this determines $\dot{I}^*_t$.

We now verify that equations \eqref{varcx} and \eqref{1,1 cond} hold---these being the differential versions of the integral conditions \eqref{one} and \eqref{two}.  Equation \eqref{1,1 cond}, i.e., that $\dot{B}_t = i dz \^ d\bar{z}$ is of type $(1,1)$, is clear.  For equation \eqref{varcx}, we actually check the dual version, $\dot{I}_t^* = \dot{B}_t Q$:
\begin{align*}
\dot{B}_t Q \,dw_t &= (i dz \^ d\bar{z}) \left( 2i(\sigma - \bar\sigma) \, (dw_t) \right) \\
&= -2 (dz \^ d\bar{z}) (w \del_w \^ \del_z - \bar{w} \del_{\bar{w}} \^ \del_{\bar{z}}) (e^{it\bar{z}} dw + it w_t d\bar{z}) \\
&= -2 w_t d\bar{z} = \dot{I}^*_t dw_t,
\end{align*}
and $\dot{B}_t Q \,dz = 0 = \dot{I}^*_t dz$.  Since $I_t$, $B_t$ and $Q$ satisfy equations \eqref{varcx} and \eqref{1,1 cond}, by integrating we see that they also satisfy equations \eqref{one} and \eqref{two}.  Therefore, these data determine a family of gauge-equivalent holomorphic Poisson structures.

As in Section \ref{holomorphic from gauge section}, we take a potential function $f = z\bar{z}$, so that $\dot{B}_t = i\bar\del \del f$.  We find the corresponding real-Hamiltonian vector field:
\begin{align*}
X &= Qdf \\
&= 2i\left(w \del_w \^ \del_z - \bar{w} \del_{\bar{w}} \^ \del_{\bar{z}}\right) (zd\bar{z} + \bar{z} dz) \\
&=  -2i\left(w \bar{z} \del_w - \bar{w} z \del_{\bar{w}}\right) \\
&= 4\Im(w\bar{z} \del_w)
\end{align*}
This is precisely the vector field that generates the family of diffeomorphisms taking $w$ to $w_t = w e^{it\bar{z}}$.

\section{Analytic structure of the complex locus}\label{analytic structure}
We first recall the holomorphic version of the notion of a scheme in algebraic geometry.
\begin{defn}
A \emph{complex analytic space} is a ringed space $(X,\O_X)$ that is locally isomorphic to the zero locus of a finite set of holomorphic functions in finitely many variables, equipped with the quotient sheaf of the ideal generated by these functions.
\end{defn}

Let $(M,\JJ)$ be a generalized complex manifold, and let $X \subset M$ be its \emph{complex locus}, consisting of the points where the Poisson structure $Q$ vanishes, and hence where $\JJ$ has the form~\eqref{usuform} of a usual complex structure. 

By Theorem~\ref{bail}, each point $p\in X$ has a neighbourhood $U$ in which $\JJ$ is gauge-equivalent to a holomorphic Poisson structure $(I,\sigma)$, with $I$ a complex structure on $U$ and $\sigma$ a holomorphic Poisson structure on $U$ such that 
\[
\sigma = IQ|_U + iQ|_U.
\]
The complex locus $X_U=X\cap U$ then coincides with the vanishing locus of the holomorphic section $\sigma$, and so inherits a complex analytic space structure in which
\begin{equation}\label{locstr}
\O_{X_U} = \O_U/\I_{X_U},
\end{equation}
where $\O_U$ is the sheaf of $I$-holomorphic functions on $U$ and $\I_{X_U}$ is the vanishing ideal of $\sigma$, defined as the image sheaf of $\sigma$ acting on holomorphic two-forms:
\[
\sigma: \Omega^2_U\to \O_U.
\] 

\begin{thm}\label{schem}
The complex locus $X$ naturally inherits the structure of a complex analytic space, such that if $\JJ$ is realized locally as a holomorphic Poisson structure $\sigma$, then the complex analytic space structure on $X$ coincides with that on the vanishing locus of $\sigma$.
\end{thm}
\begin{proof}
We demonstrate this by showing that the structure sheaf~\eqref{locstr} is independent of the choice of local realization of $\JJ$ as a holomorphic Poisson structure.  More precisely, we show that if $U$ and $V$ are neighbourhoods as above, in which $\JJ$ is gauge equivalent to the holomorphic Poisson structures $(I,\sigma)$, $(J, \tau)$ respectively, with corresponding structure sheaves $\O_{X_U}, \O_{X_V}$ as in~\eqref{locstr}, then there is a canonical sheaf isomorphism 
\[
\varphi_{V,U}:\O_{X_U}|_{U\cap V} \to \O_{X_V}|_{U\cap V},
\]
which satisfies the gluing condition
\begin{equation}\label{cocycl}
\varphi_{W,V}\circ \varphi_{V,U} = \varphi_{W,U}
\end{equation}
for any triple of neighbourhoods $U,V,W$ as above.  
Finally, we may cover $X$ by open sets $\{U_i\}$ of the above form, and apply the gluing theorem for sheaves \cite[\href{http://stacks.math.columbia.edu/tag/00AK}{\S 6.33}]{stacks-project}
to the local gluing data $\{\O_{X_{U_i}},\varphi_{U_j,U_i}\}$ to obtain the required structure sheaf $\O_X$ of the complex locus.

\vspace{1ex}

To construct $\varphi_{V,U}$, let $p\in U\cap V$, and let $\tilde\varphi_p$ be a Hamiltonian diffeomorphism, as in Theorem~\ref{hamthm}, which defines an isomorphism of holomorphic Poisson structures from $(I,\sigma)|_{U_p}$ to $(J,\tau)|_{\varphi_p(U_p)}$, where $U_p\subset U\cap V$ is a neighbourhood of $p$. Recall that $\tilde\varphi_p$ is the time-1 flow of a Hamiltonian vector field for the real Poisson structure $Q$, hence it fixes $X_{p} = X\cap U_p$ pointwise, and since it takes $\sigma$ to $\tau$, it induces an isomorphism of sheaves
\[
\varphi_p: \O_{X_U}|_{X_{p}}\to \O_{X_V}|_{X_{p}}. 
\]  

\vspace{1ex}

We now prove that $\varphi_p$ is independent of the particular Hamiltonian flow $\tilde\varphi_p$ used to interpolate between $(I,\sigma)$ and $(J,\tau)$.  Indeed, let $\tilde\varphi_t, \tilde\varphi_t'$ ($t\in[0,1]$) be two such flows, generated by time-dependent Hamiltonians $f_t, f'_t$ respectively, and defined in a neighbourhood $U_p\subset U\cap V$ of $p$.  If $h_0$ is an $I$-holomorphic function on $U_p$, then the resulting pullbacks differ by
\[
\begin{aligned}
\Delta h_1 = {{\tilde\varphi}_1}^{\prime *} h - {\tilde\varphi}_1^* h &= \int_0^1 \left(L_{X_{f'_t}} (\tilde\phi_t^{\prime *} h_0) -L_{X_{f_t}} (\tilde\phi_t^* h_0) \right)dt\\
&= i_Q\int_0^1\left(df'_t\wedge d\tilde\phi_t^{\prime *} h_0 - df_t\wedge d\tilde\phi_t^* h_0\right)dt.
\end{aligned}
\]
Therefore, $\Delta h_1$ lies in the vanishing ideal of $Q$ in the smooth complex-valued functions on an open neighbourhood of $X_p = U_p\cap X$.  Since $Q = \Im(\sigma)$, this vanishing ideal coincides with the ideal generated by $\sigma$ and $\bar\sigma$, and we may apply Malgrange's criterion~\cite[VI, Theorem 1.1.]{MR0212575} for ideal membership to deduce that the holomorphic function $\Delta h_1$ lies in the vanishing ideal generated by $\sigma$ alone (in the smooth functions).  To complete the argument, we must show that $\Delta h_1$ lies in the ideal of $\sigma$ in the holomorphic functions, so that ${{\tilde\varphi}_1}^{\prime *} h$ and ${\tilde\varphi}_1^* h$ coincide in the structure sheaf $\O_{X_V}|_{X_p}$, proving that the induced map $\varphi_p$ is independent of the chosen flow.

Let $\sr{I}=(\sigma)$ be the vanishing ideal of $\sigma$ in the ring $\sr{R}$ of convergent holomorphic power series centered on $p$.  Since $\Delta h_1$ is in the smooth ideal generated by $\sigma$, its Taylor series about $p$ lies in the ideal $\sr{I}\hat{\sr{R}}$ generated by $\sigma$ in the completion of $\sr{R}$, i.e. the ring $\hat{\sr{R}}$ of formal power series in holomorphic coordinates centered on $p$.  
By Krull's theorem~\cite[IV,\S 7.]{MR0090581}, we have the identity
\[
\sr{R}\cap (\sr{I}\hat{\sr{R}}) = \sr{I}, 
\]
proving that $\Delta h_1$ lies in the holomorphic vanishing ideal of $\sigma$ near each point of $U\cap V$, as required.
  
The same argument may be used to show that isomorphisms $\varphi_p, \varphi_q$ defined as above in neighbourhoods $X_p, X_q$ of points $p,q\in X$ actually coincide on $X_p\cap X_q$, and so glue together to define the required sheaf isomorphism $\varphi_{V,U}$. 

\vspace{1ex}

It remains to verify the cocycle condition~\eqref{cocycl}. By the results above, in a sufficiently small neighbourhood of $p\in U\cap V\cap W$, we may express $\varphi_{V,U}$ and $\varphi_{V,W}$ as morphisms induced by Hamiltonian flows: let $\tilde\varphi_t$ be the Hamiltonian flow of $f_t$ which takes $(I,\sigma)$ to $(J,\tau)$ after unit time, and let $\tilde\psi_t$ be the flow of $g_t$, which similarly takes $(J,\tau)$ to $(K,\mu)$. The composition of flows $\tilde\rho_t =\tilde\psi_t \circ\tilde\varphi_t$ is also a $Q$-Hamiltonian flow, for the time-dependent Hamiltonian 
\[
h_t = g_t + (\psi_t)_* f_t.
\]
By definition, the unit time flows satisfy  $\tilde\rho_1 =\tilde\psi_1 \circ\tilde\varphi_1$, and so we obtain $\rho_p=\psi_p\circ\varphi_p$ for the induced sheaf isomorphisms from $\O_{X_U}$ to $\O_{X_W}$ in a sufficiently small neighbourhood of each point $p\in U\cap V\cap W$, as required.
% Using the path of complex structures $I'_t = \rho(t)_* I$, we define the closed 2-form
% \[
% F_t = \int_0^t (dd^c_{I'_s}h_s) ds,
% \]
% so that the path of Courant symmetries $(\rho(t), F_t)$ interpolates between the holomorphic Poisson structures $(I,\sigma)$ and $(K,\mu)$.  Flowing for unit time, we see that \\
% Use the notation $I_{t,s} = (\psi(s)\varphi(t))_* I$, so that $I_{0,0} = I$ and $I_{1,1} = K$.  This implies that 
% \[
% \begin{aligned}
% I_{t,0} - I &= Q F_{t,0}\\
% I_{t,s} - I_{t,0} &= Q G_{t,s},
% \end{aligned}
% \]
% where the closed 2-forms $F,G$ are given by  
% \[
% F_{t,0} = \int_0^t (dd^c_{I_{u,0}} f_u) du,\qquad {G}_{t,s} = \int_0^s (dd^c_{I_{t,u}} g_u) du.
% \]
% As a result, we have that 
% \[
% I_{t,t} - I = QH_t,\qquad H_t = G_{t,t} + F_{t,0}.
% \]
% We now choose a smooth function $h_t$ {\color{red} dependence on t?} such that 
% \[
% dd^c_{I_{t,t}}h_t = dd^c_{I_{t,0}} f_t+dd^c_{I_{t,t}} g_t.
% \]
% The function $h_t$ will then\footnote{Could possibly cite my branes paper for this result - groupoid part.} generate a Hamiltonian flow taking $(I,\sigma)$ to $(K,\mu)$, and this flow coincides with the composition $\psi(t)\varphi(t)$, showing that the composition $\phi(1)\varphi(1)$ of the flows used to define $\varphi_{W,V}$ and $\varphi_{V,U}$ coincides with a flow which induces $\varphi_{W,U}$, proving the cocycle condition.
\end{proof}

\pagebreak
\appendix

\section{Normal form for families of generalized complex structures}\label{SCI section}

The purpose of this appendix is to sketch the proof of Theorem
\ref{families prop}, which is a parametrized version of Theorem
\ref{bail}.  We briefly review the methods used in the proof of Theorem
\ref{bail} and discuss how the argument passes to families.  We shall not reproduce
technical details which can be found in \cite{Bailey}.

\subsection{The SCI framework and the normal form lemma}

The proof of Theorem \ref{bail} in \cite{Bailey} uses a general
technical lemma which enables one to show that a given geometric
structure is equivalent to one in ``normal form'' in a suitably small
neighbourhood of a point $p$.
The lemma, \cite[Theorem 4.17]{Bailey}, is an extension of the results
of Miranda, Monnier and Zung \cite{MR2855089}, which encapsulate a
technique used by Conn \cite{MR794374} in the linearization of Poisson
structures, which was itself a version of the Nash-Moser fast
convergence technique adapted to spaces of local sections about a
point.

To do this, one shows that there is a local automorphism of the space
which takes the original structure to one which is
\emph{approximately} in normal form; by iterating this approximation
one finds, in the limit, a local automorphism taking the original
structure to one precisely in normal form.  To establish the limit one
uses the technique of Nash and Moser \cite{MR656198}, except that one
must take care of the fact that, since one is working in
neighbourhoods of $p$, at each step the automorphism may necessitate
restricting to a smaller neighbourhood.  With suitable
estimates, one controls how quickly the neighbourhood shrinks in this
iteration, and shows that in the limit one gets a neighbourhood of
positive radius.  This innovation is due to Conn \cite{MR794374}.

Miranda, Monnier and Zung develop the framework of \emph{SCI-spaces},
short for ``scaled $C^\infty$'' spaces.  We briefly summarize the
framework and the lemma, focusing on the modifications required 
for the generalization to families.
% including only enough detail to satisfy the
% reader that our generalization to families is valid.  
For full details, consult \cite[Section 4]{Bailey} and
\cite{MR2855089}.  For the general theory of tame Fr\'{e}chet spaces,
smoothing operators and the Nash-Moser technique, consult
the notes of Hamilton \cite{MR656198}.

\subsubsection{SCI-spaces}

An \emph{SCI-space} $\sr{V}$ is a radius-parametrized collection of
\emph{tame Fr\'{e}chet spaces}.  That is, for each $r \in (0,1]$ there
is a Fr\'{e}chet space $\sr{V}_r$ with a nondecreasing sequence of
norms $\|\cdot\|_{0,r}\, ,\, \|\cdot\|_{1,r}\, ,\, \|\cdot\|_{2,r}\,
,\, \ldots$ and \emph{smoothing operators} $S_r(t)$ for real $t>1$;
the smoothing operators must satisfy certain well-known estimates.
Furthermore, there is a \emph{radius restriction map} from $\sr{V}_r$
to $\sr{V}_{r'}$ whenever $r \geq r'$, and all diagrams of restriction
maps commute.  Using these restrictions, we may
identify a vector $v$ in $\sr{V}_r$ with its preimages at larger
radii.  Finally, we impose the condition that the norms $\|v\|_{k,r}$
are nondecreasing in $r$ (as well as in $k$).  When the radius is
clear from context, we often omit it and simply write $\|v\|_k$.

The prototypical example of an SCI-space is given by the local sections of a
vector bundle $V$ with connection about a point $p$ in a Riemannian
manifold.  Each $\sr{V}_r$ consists of the smooth sections of $V$
restricted to a closed ball of radius $r$ centered at $p$, equipped
with the usual $C^k$ norms.  A typical construction of smoothing operators $S(t)$ on, eg., the space of smooth, compactly supported functions on $\R^n$ is to Fourier transform the function, remove frequencies higher than $\tfrac{1}{t}$ by multiplying by a cutoff function, and then transforming back to position space.  Such operators can be transferred to the space of sections of a vector bundle on a manifold through the use of embeddings.

% A map between Fr\'{e}chet spaces is \emph{tame} when it is continuous
% from the $k$-norm to the $k-d$-norm   
% and there are standard constructions of smoothing operators.  The
% proof of Theorem \ref{bail} in \cite{Bailey} used the SCI-space of
% sections of the generalized complex deformation bundle for a complex
% structure near a point.

\subsubsection{SCI-groups and actions}

% We have a notion of a group-like structure modelled on an SCI-space.
An \emph{SCI-group} $\sr{G}$ is an SCI-space $\sr{W}$ intended to
model local diffeomorphisms; $\sr{W}$ is equipped with an associative
partial composition as well as an identity element $\Id$.  Whenever
$\|\phi-\Id\|_{1,r}$ and $\|\psi-\Id\|_{1,r}$ are small enough (where
the bound depends on $r$), the product $\phi\cdot\psi$ is well-defined
(at a certain radius $r'<r$).  The product commutes with restriction,
and inverses exist, but once again only at a smaller
radius, and only if $\|\phi-\Id\|_{1,r}$ is small enough.  Finally, the
product and inverse operations must satisfy certain norm estimates
(given in \cite{Bailey}).
% We remark that, though $\sr{W}$ models a ``neighbourhood of the
% identity'' in $\sr{G}$, in no way should it be interpreted as the
% Lie algebra for $\sr{G}$.

A typical example of an SCI-group is given by the local
diffeomorphisms fixing the origin in $\R^n$.  In this case, we may
take $\sr{W}_r$ to be the smooth functions $\chi$ from the closed ball
of radius $r$ to $\R^n$ which vanish at the origin and are such that
$\phi = \Id + \chi$ is a local diffeomorphism.  
The example which we use in the proof of
Theorem \ref{bail} is the SCI-group of local Courant automorphisms,
which consist of diffeomorphisms composed with $B$-field gauge
transformations.

There is also a notion of the action of an SCI-group on an SCI-space,
with a similar accounting for radius restrictions, and also satisfying
tameness estimates.  The typical example of an SCI-action is the
action of local diffeomorphisms on local tensor fields by pushforward
or pullback.  In the proof of Theorem \ref{bail}, we use the action of
local Courant automorphisms on local deformations of generalized
complex structure.

\subsubsection{The normal form lemma}

We now present an outline of the main lemma \cite[Theorem
4.17]{Bailey}, including the intended interpretations of the spaces
and maps involved. The lemma is applied in a context where the
geometric structures in question are described as sections of a vector
bundle satisfying an integrability condition such as the Maurer-Cartan
equation. We refer to such sections as \emph{pre-integrable} if we do
not impose the integrability condition. For simplicity, we assume that
the desired normal form of the geometric structure may be expressed as
a constraint on the pre-integrable section, followed by imposing the
integrability condition.  Note that we are using ``normal form'' to
mean a geometric structure satisfying a constraint, rather than having
a fixed representation in local coordinates.

\begin{lem}\label{normal form lemma}
Suppose we have the following SCI-spaces:
\begin{itemize}
\item[--] $\sr{T}$ (the pre-integrable geometric structures),
\item[--] $\sr{F} \subset \sr{T}$ (the pre-integrable structures in normal form),
\item[--] $\sr{I} \subset \sr{T}$ containing $0$ (the integrable structures),
\item[--] $\sr{N} = \sr{F} \cap \sr{I}$ (the integrable structures in normal form), and
\item[--] $\sr{V}$ (the infinitesimal automorphisms),
\end{itemize}
and let $\sr{G}$ be an SCI-group (the local automorphisms) which acts
on $\sr{T}$, preserving $\sr{I}$.  Let $\pi : \sr{T} \to \sr{F}$ be a
projection, and define $\zeta = \Id-\pi$, which measures the
failure to be in normal form.  Suppose we have maps
$$\sr{I} \oto{V} \sr{V} \oto{\Phi} \sr{G},$$
where $V$ provides an infinitesimal automorphism whose time-1 flow,
given by $\Phi$, should bring a given structure closer to normal form.

Suppose furthermore that these maps satisfy the set of estimates given
in \cite[Theorem 4.17]{Bailey}, including in particular that there is
some $\delta>0$ and $s\in\mathbb{N}$ such that, for any $\epsilon \in
\sr{I}$,
\begin{align}\label{quadratic estimate}
  \left\|\zeta\left(\Phi_{V(\epsilon)}\cdot\epsilon\right)\right\|_k
  \leq \|\zeta(\epsilon)\|^{1+\delta}_{k+s} E
  \left(\zeta(\epsilon),\epsilon , \Phi_{V(\epsilon)}-\Id
    \right),
\end{align}
where $E$ is a polynomial in the $(k+s)$-norms of its arguments with
positive coefficients.

Then there exists $l \in \N$ and real positive constants $\alpha,\beta$, such
that for any $\epsilon \in \sr{I}_R$, if $\|\epsilon\|_{2l-1,R} <
\alpha$ and $\|\zeta(\epsilon)\|_{l,R} < \beta$ then there exists
$\psi \in \sr{G}$ such that $\psi \cdot \epsilon \in \sr{N}_{R/2}$.
\end{lem}

\begin{rem}
  The omitted estimates relate to the Fr\'echet \emph{tameness} of the various
  maps.  We highlight estimate \eqref{quadratic estimate} because it
  expresses the key fact that the iteration
$$\epsilon \mapsto \Phi_{V(\epsilon)} \cdot \epsilon$$
shrinks the error ``quadratically'' (strictly speaking, by the power
$(1+\delta)$).  The estimates, including \eqref{quadratic estimate},
have the property that some derivatives are lost, i.e., the right hand
sides involve higher derivative norms.  To ensure convergence in spite
of this problem, the naive iteration is modified as per Nash-Moser by
applying smoothing operators to $V(\epsilon)$ at each stage.
\end{rem}

\subsubsection{Application to generalized complex structures}\label{gcsnorm}
We now describe how Lemma~\ref{normal form lemma} is used
to prove Theorem~\ref{bail}. 

Given a point of complex type on a generalized complex manifold, a
scaling argument~\cite[Section 7]{Bailey} shows that this point has
neighbourhoods in which the generalized complex structure is
equivalent to arbitrarily small generalized complex deformations of
the standard complex structure on some ball $B_r\subset \C^n$, with
the property that the deformation is trivial at the center of the ball.

To establish Theorem~\ref{bail}, it remains to show that an
arbitrarily small such deformation of $B_r\subset \C^n$ is equivalent
(on a possibly smaller ball) to a holomorphic Poisson structure.  It
is at this point that the Nash--Moser iteration scheme on SCI spaces
is invoked~\cite{Bailey}.  A generalized complex deformation
$\epsilon$ of the usual complex structure on $\C^n$ is a section with
three components \cite[Section 2.2]{Bailey}:
\begin{align*}
\epsilon^{2,0} &\in C^\infty( \wedge^2 T_{1,0}), \\
\epsilon^{1,1} &\in C^\infty( T_{1,0} \otimes T^*_{0,1} ), \\
\epsilon^{0,2} &\in C^\infty(\wedge ^2 T^*_{0,1} ),
\end{align*}
satisfying a Maurer-Cartan equation which generalizes the one
governing deformations of complex structure.  
The relevant SCI spaces,
defined on closed balls $B_r$ centered at $0\in\C^n$, are defined as
follows:
\begin{itemize}
\item[--] $\sr{T}$ consists of pre-integrable deformations $\epsilon$ of the above form.
\item[--] $\sr{I}\subset\sr{T}$ consists of the integrable (Maurer-Cartan) deformations.
\item[--] $\sr{F}$ are pre-integrable deformations of bivector type, i.e. with vanishing  $\epsilon^{1,1}$ and $\epsilon^{0,2}$, so that $\sr{N}=\sr{F}\cap\sr{I}$ are holomorphic Poisson tensors.
\item[--] $\sr{G}$ are Courant automorphisms fixing $0\in\C^n$.
\item[--] $\sr{V}$ are the generalized vector fields $C^\infty(T\oplus T^*)$, and $\Phi$ is their time-1 flow, described in~\cite[Section 2.3]{Bailey}.
 \end{itemize}

% We will sketch how this theorem was applied in the proof of Theorem
% \ref{bail}.  $\sr{T}$ consisted of the generalized complex
% deformations of the complex structure on closed balls $B_r$ about $0
% \in \C^n$, $\sr{I} \subset \sr{T}$ was the integrable (Maurer-Cartan)
% deformations, $\sr{F}$ was the purely Poisson deformations, and thus
% $\sr{N}=\sr{F}\cap\sr{I}$ was the holomorphic Poisson deformations.
% $\sr{G}$ was the local generalized diffeomorphisms fixing $0 \in
% \C^n$, $\sr{V}$ was the ``generalized vector fields'' $C^\infty(TB_r
% \dsum T^*B_r)$, and $\Phi_\bullet$ was the time-1 flow of these.

% A generalized complex deformation $\epsilon$ of $\C^n$ may be split into three parts \cite[Section 2.2]{Bailey},
% \begin{align*}
% \epsilon^{2,0} &\in C^\infty(\^ ^2 T^{1,0} B_r) \\
% \epsilon^{1,1} &\in C^\infty \left( T^{1,0} B_r \^ (T^*)^{0,1} B_r \right) \\
% \epsilon^{0,2} &\in C^\infty(\^ ^2 (T^*)^{0,1} B_r)
% \end{align*}
 The projection map $\pi:\sr{T}\to \sr{F}$ is defined by
 $\epsilon\mapsto \epsilon^{2,0}$, and so
 $\zeta(\epsilon)=\epsilon^{1,1}+\epsilon^{0,2}$. On $B_r$, the
 Dolbeault complex admits a homotopy operator $P$, so that $\Id = P
 \comp \bar\del + \bar\del \comp P$.  This may be used to define the
 infinitesimal correction
\begin{align}\label{correction formula}
V(\epsilon) = P\left(\left[\epsilon^{2,0},P\epsilon^{0,2}\right] - \epsilon^{1,1} - \epsilon^{0,2}\right),
\end{align}
where the bracket in \eqref{correction formula} is the natural
extension of the Schouten bracket to the Dolbeault complex with
coefficients in holomorphic multivector fields.  One then shows
\cite[Lemma 6.11]{Bailey} that the key estimate~\eqref{quadratic
  estimate} holds.

% In the limit where the error $\zeta(\epsilon) = \epsilon^{1,1} +
% \epsilon^{2,2}$ is small, the infinitesimal action of $V(\epsilon)$ on
% $\epsilon$ approximately cancels $\epsilon^{1,1}$ and
% $\epsilon^{2,2}$ \cite[Lemma 6.11]{Bailey}.

Given that this and the other tameness estimates are satisfied,
Lemma~\ref{normal form lemma} implies that for any sufficiently small
generalized complex deformation of $B_r \subset \C^n$ which vanishes
at $0$, there is a neighbourhood of zero in which it is equivalent, by
a Courant automorphism, to a holomorphic Poisson structure.

% As shown in \cite[Section 7]{Bailey} by a scaling argument, a point of
% complex type in a generalized complex manifold always has
% neighbourhoods equivalent to arbitrarily small deformations of
% $B_r\subset \C^n$, so that Theorem \ref{bail} follows.

\subsection{Application to families of generalized complex structures}\label{families algorithm}

We now explain how to extend the results of Section~\ref{gcsnorm} in
such a way that a smooth family of generalized complex structures
parametrized by $S=[0,1]$ which have complex type at a point $p\in M$
is seen to be equivalent, in a sufficiently small neighbourhood of
$p$, to a smooth family of holomorphic Poisson structures, proving
Theorem~\ref{families prop}.  We are particularly interested in the
case where the given family of structures is already of Poisson type
at the boundary $\{0,1\}$ of the parameter space, in which case we
want these to be fixed by the equivalence.

\vspace{.4em}
\noindent{\bf Working in families}
\vspace{.4em}

%
%Given a family of generalized complex structures parametrized by $S=[0,1]$ which have complex type at a point $p \in M$, we first apply a it may be the case that the induced complex structures at $p$ differ among the family.  Nonetheless, we may rectify this difference by applying a family of local diffeomorphisms about $p$ so that all induced complex structures at $p$ are the same, and with a family of $B$-transforms we can ensure furthermore that all the generalized complex structures at $p$ are the same.  In this way, and by choosing coordinates, without loss of generality we may consider only the case where $M = \C^n$ and where every generalized complex structure in the family agrees at $p = 0 \in C^n$ with the standard generalized complex structure on $\C^n$ .

Let $X = \C^n\times S$, with projection $\pi:X\to S$ to the parameter
space $S=[0,1]$, and let $V = \ker \pi_*$ be the relative tangent
bundle. To describe geometric structures parametrized by $S$, we use the
relative Courant algebroid $V\oplus V^*$ over $X$.  A family of
generalized complex structures parametrized by $S$ is a complex
structure on the bundle $V\oplus V^*$ which is integrable with respect
to the vertical Courant bracket. 

The given family of generalized complex
structures defines precisely such a structure in a neighbourhood of
the zero section $\{0\} \times X$.  The relevant symmetries for such families
are the relative Courant automorphisms, generated by diffeomorphisms
$\varphi$ of $X$ such that $\pi\varphi = \pi$, together with
$B$-field transformations by relatively closed vertical 2-forms.

We first apply a symmetry to ensure that the family of generalized complex structures is constant along the zero section $\{0\}\times X$, and agrees with the standard complex structure on $\C^n$ along this locus.  
Then the same scaling argument from \cite[Section 7]{Bailey}, applied
vertically to the family $X\to S$, shows that any family of
generalized complex structures parametrized by $S=[0,1]$, each member of
which agrees at $p$ with the standard structure on $\C^n$, there exists a tubular neighbourhood, $B_r\times S$, of
$\{0\}\times S$ ($B_r \subset \C^n$ being a closed ball about $0$) on which it is equivalent to an arbitrarily small family, $\epsilon$, of generalized complex deformations of the constant family of standard complex structures
on $\C^n$.  Furthermore, $\epsilon$ vanishes along the zero section $\{0\}\times S$.  

\vspace{.4em}
\noindent{\bf SCI spaces for families}
\vspace{.4em}

As in Section~\ref{gcsnorm}, what remains is to show that an
arbitrarily small such deformation is equivalent (on a possibly
smaller tubular neighbourhood) to a family of holomorphic Poisson
structures.  We now explain how the Nash--Moser iteration scheme on
SCI spaces can be adapted to this situation.

The SCI spaces used in Lemma~\ref{normal form lemma} were constructed
as spaces of sections of vector bundles over closed balls $B_r\subset
\C^n$.  To work in families over $S$, we pull back these vector
bundles to $X_r=B_r\times S$ and consider their sections over the
total space.  To ensure that our families are smooth, we use $C^k$
norms over the total space $X$, rather than only over $B_r$ as done in
\cite{Bailey}.  We set up the SCI spaces as follows, with restrictions
to $X_r$ understood:
\begin{itemize}
\item[--] $\sr{T}$ consists of relative pre-integrable deformations 
  \begin{equation*}
    \epsilon\in C^\infty(X, \wedge^2 (V_{1,0}\oplus V^*_{0,1})),
  \end{equation*}
  where $V_{1,0}$ and $V_{0,1}$ are the $+i$ and $-i$ eigenbundles,
  respectively, of the standard complex structure on the fibres
  of $\pi$.
\item[--] $\sr{I}\subset\sr{T}$ are the above sections which satisfy
  the Maurer-Cartan equation.
\item[--] $\sr{F}$ are pre-integrable relative deformations of
  bivector type, with $\epsilon^{1,1}=\epsilon^{0,2}=0$, so that
  $\sr{N}=\sr{F}\cap\sr{I}$ are smooth families of holomorphic Poisson
  tensors,
\item[--] $\sr{G}$ are the local relative Courant automorphisms fixing
  $\{0\}\times S\subset X$ and which are trivial on the restriction of $V\oplus V^*$ to this locus,
\item[--] $\sr{V}$ are the infinitesimal relative symmetries given by
  sections $C^\infty(V\oplus V^*)$, and $\Phi$ is their time-1 flow.
 \end{itemize}
 The projection map $\pi:\sr{T}\to\sr{F}$ is defined as before, and
 the infinitesimal correction operator $V:\sr{T}\to\sr{V}$ is 
 defined by~\eqref{correction formula}, where we view $P$ as the homotopy
 operator for the vertical Dolbeault complex.

 In order to study families of deformations which are already in
 normal form at the boundary of the parameter space, we must consider
 the following subspaces of the SCI spaces defined above. First we let
 $\sr{T}_\del\subset \sr{T}$ be the deformations $\epsilon$ for which
 the error $\zeta(\epsilon)=\epsilon^{1,1} + \epsilon^{0,2}$ vanishes
 on $B_r\times \del S$.  The appropriate symmetries in this case form
 an SCI subgroup $\sr{G}_\del\subset \sr{G}$ of automorphisms which
 restrict to the identity at $B_r\times\del S$.  The corresponding
 infinitesimal symmetries are then the sections
 $\sr{V}_\del\subset\sr{V}$ which vanish on $B_r\times\del S$.  The
 maps $\Phi$ and $V$ defined above have well-defined restrictions to
 these subspaces, defining maps
\begin{equation*}
\sr{I}_\del \oto{V} \sr{V}_\del \oto{\Phi} \sr{G}_\del.
\end{equation*}
The Nash--Moser iteration requires smoothing operators for the SCI
spaces defined above.  To obtain these for the manifold with boundary
$X$, we may apply standard ``doubling'' arguments (see
\cite[II.1.3]{MR656198}).  So, in order to carry over the results of
Lemma \ref{normal form lemma}, we need only check the required
estimates for the various maps defined above. After this verification
(which we outline in Section~\ref{estim}), we may conclude that the
iteration provides a smooth family of automorphisms rendering each of
the given generalized complex structures into a holomorphic Poisson
structure, establishing Theorem~\ref{families prop}.

% \begin{thm}[Thm. \ref{families prop}]
%   Let $\JJ_t$, $t\in [0,1]$ be a smooth family of generalized complex
%   structures on a manifold $M$ that are all of complex type at the
%   point $p\in M$.  Then, in a neighbourhood of $p$, there is a smooth
%   family of gauge transformations by closed 2-forms $B_t$ which
%   renders $\JJ_t$ isomorphic to a smooth family of holomorphic Poisson
%   structures $(I_t,\sigma_t)$ with $\sigma_t(p)=0$ for all~$t$.

%   Furthermore, if $\JJ_0$ and $\JJ_1$ are already of holomorphic
%   Poisson type~\eqref{holpo}, then the family $B_t$
%   may be chosen such that $B_0=B_1=0$.
% \end{thm}

\subsubsection{Estimates}\label{estim}

The proof of Theorem~\ref{bail} involves two groups of estimates on
$C^k$ norms of tensors over $B_r\subset \C^n$. The first group of
estimates, given in Lemmas 5.6--5.12 in \cite{Bailey}, establish that
the local Courant automorphisms $\sr{G}$ form an SCI group, and that
their action on the deformations $\sr{T}$ defines an SCI group
action. By the same arguments presented in~\cite{Bailey}, the same is
true for relative Courant automorphisms and families of deformations.
The second group of estimates, given in Lemmas 6.1--6.7 in
\cite{Bailey}, establish the necessary properties of the maps
$V,\zeta,\Phi$ among the SCI spaces. This last group includes the key
estimate~\ref{quadratic estimate}. To establish these for families
requires straightforward modifications to take into account
derivatives in all directions in the total space $X = B_r \times S$
rather than just the vertical ones.  To illustrate this we provide an
example, showing how \cite[Lemma 6.2]{Bailey} is extended to the
families setting.

Let $\JJ$ be an involutive complex structure on $V\oplus V^*$ over a
neighbourhood of the zero section in $X$, representing a family of
generalized complex structures near the origin in $\C^n$ parametrized
by $S$, and let $L\subset (V\oplus V^*)\otimes\C$ be its $+i$
eigenbundle, which is a Lie algebroid using the Courant bracket.  To
extend~\cite[Lemma 6.2]{Bailey}, we need a bound for the induced
Schouten bracket on sections of $\wedge^\bullet L$.  For
$\alpha$, $\beta$ sections of $\wedge^\bullet L$, the bracket
$[\alpha,\beta]$ is a pointwise bilinear function of the
vertical 1-jets of $\alpha$ and $\beta$, and so \emph{a fortiori} it is a 
pointwise bilinear function of their full 1-jets on $X$.  Differentiating 
using the product rule, we obtain the bound 
\[
\|[\alpha,\beta]\|_{k} \leq C_k \|\alpha\|_{k+1}\|\beta\|_{k+1},
\]
where all norms are now $C^k$ norms over all of $X$. The right hand side 
is of type $\LP(\|\alpha\|_{k+1},\|\beta\|_{k+1})$ in the notation of \cite{Bailey}, as required for the SCI formalism.

\bibliographystyle{hyperamsplain}
\bibliography{gc local}

\end{document}